\documentclass[12pt]{article}
\usepackage{fullpage,enumerate}
\usepackage[utf8]{inputenc}
\usepackage[intlimits]{amsmath}
\usepackage{amssymb,latexsym,amsthm,mathtools,tikz}
\usepackage[hidelinks]{hyperref}
\usepackage[textmode]{ytableau}
\usepackage[font=footnotesize,labelfont=bf]{caption}

\allowdisplaybreaks[3]
\everymath{\displaystyle}

\newcommand{\smallfrac}[2]{\textstyle\frac{#1}{#2}\displaystyle}

\newcommand{\CC}{\mathcal{C}}
\newcommand{\DD}{\mathcal{D}}
\newcommand{\CChat}{\widehat{\CC}}
\newcommand{\rc}[1]{rc(#1)}
\newcommand{\inv}[1]{{#1}^{rc}}
\newcommand{\ind}[1]{#1^{\not\oplus}}
\newcommand{\indinv}[1]{#1^{\not\oplus rc}}
\newcommand{\gr}[1]{\textup{gr}(#1)}
\newcommand{\upgr}[1]{\overline{\textup{gr}}(#1)}
\newcommand{\logr}[1]{\underline{\textup{gr}}(#1)}
\newcommand{\grrc}[1]{\textup{gr}^{rc}(#1)}
\newcommand{\upgrrc}[1]{\overline{\textup{gr}}^{rc}(#1)}
\newcommand{\logrrc}[1]{\underline{\textup{gr}}^{rc}(#1)}

\newcommand{\sumof}{{\textstyle\bigoplus}}

\newcommand{\Av}[1]{\textup{Av}(#1)}
\newcommand{\Avn}[2]{\textup{Av}_{#1}(#2)}

\newcommand{\geom}[1]{\textup{Geom}(#1)}
\newcommand{\geomn}[2]{\textup{Geom}_{#2}(#1)}
\newcommand{\Geom}[1]{\textup{Geom}\!\left(#1\right)}
\newcommand{\griddings}[1]{\textup{Geom}^\sharp(#1)}
\newcommand{\griddingsn}[2]{\textup{Geom}^\sharp_{#2}(#1)}

\newcounter{i}

\newcommand{\drawpermutation}[3][1]{\begin{tikzpicture}[scale=0.5,baseline=(O.base)]
\setcounter{i}{0}
\foreach \j in {#2} {
\stepcounter{i}
\draw (0.5*#1,\value{i}*#1) -- (#3*#1+0.5*#1,\value{i}*#1);
\draw (\value{i}*#1,0.5*#1) -- (\value{i}*#1,#3*#1+0.5*#1);
\node at (\value{i}*#1, 0) {\footnotesize$\j$};
\draw[fill] (\value{i}*#1, \j*#1) circle (0.2);
}
\node (O) at (#3*0.5*#1,#3*0.5*#1) {};
\end{tikzpicture}}

\newcommand{\drawpattern}[4][1]{\begin{tikzpicture}[scale=0.5,baseline=(O.base)]
\foreach \x in {1,...,#3} {
\draw (0.5*#1,\x*#1) -- (#3*#1+0.5*#1,\x*#1);
\draw (\x*#1,0.5*#1) -- (\x*#1,#3*#1+0.5*#1);
}
\setcounter{i}{0}
\foreach \j in {#2} {
\stepcounter{i}
\node at (\value{i}*#1, 0) {\footnotesize$\j$};
\draw[fill] (\value{i}*#1, \j*#1) circle (0.2);
\foreach \k in {#4} {
\ifnum \j=\k
\draw [thick] (\value{i}*#1, \j*#1) circle (0.4);
\fi
}
}
\node (O) at (#3*0.5*#1,#3*0.5*#1) {};
\end{tikzpicture}}

\theoremstyle{theorem}
\newtheorem{thm}{Theorem}[section]
\newtheorem{prop}[thm]{Proposition}

\newtheorem{cor}[thm]{Corollary}
\newtheorem{conj}[thm]{Conjecture}
\newenvironment{restate}[1]{\theoremstyle{theorem}\newtheorem*{blah}{Theorem \ref{#1}} \begin{blah}}{\end{blah}}
\newenvironment{restate2}[1]{\theoremstyle{theorem}\newtheorem*{blah2}{Theorem \ref{#1}} \begin{blah2}}{\end{blah2}}

\theoremstyle{definition}
\newtheorem{defn}[thm]{Definition}

\begin{document}

\title{On the centrosymmetric permutations in a class}
\author{Justin M. Troyka\thanks{Department of Mathematics,
Dartmouth College, Hanover, NH 03755, United States. \href{mailto:jmtroyka@math.dartmouth.edu}{\texttt{jmtroyka@gmail.com}}.}}

\maketitle

\begin{abstract}
A permutation is centrosymmetric if it is fixed by a half-turn rotation of its diagram. Initially motivated by a question by Alexander Woo, we investigate the question of whether the growth rate of a permutation class equals the growth rate of its even-size centrosymmetric elements. We present various examples where the latter growth rate is strictly less, but we conjecture that the reverse inequality cannot occur. We conjecture that equality holds if the class is sum closed, and we prove this conjecture in the special case where the growth rate is at most $\xi \approx 2.30522$, using results from Pantone and Vatter on growth rates less than $\xi$. We prove one direction of inequality for sum closed classes and for some geometric grid classes. We end with preliminary findings on new kinds of growth-rate thresholds that are a little bit larger than $\xi$.
\end{abstract}

\section{Introduction}

This paper concerns the enumeration of the permutations in a class that are fixed by the reverse--complement transformation. This is the same kind of endeavor carried out by \cite{BHPV} with permutations fixed by a different transformation, namely taking the inverse.

We begin with terms and notation about permutation classes and their growth rates (Section \ref{subsec:terms}); the introduction continues by defining less standard terms and notation on the reverse--complement map and centrosymmetric permutations (Section \ref{subsec:rc}) and providing a summary of the main ideas of the paper (Section \ref{subsec:mainideas}).

\subsection{Permutation classes} \label{subsec:terms}

This subsection is a quick overview of ideas and notation that are standard in permutation patterns. For more background on this topic, see the survey by Vatter \cite{survey}.

The diagram of a permutation $\pi$ of size $n$ is the plot of the points $(i, \pi(i))$ for $i \in [n]$. A permutation $\pi$ \emph{contains} another permutation $\sigma$ (as a pattern) if the diagram of $\sigma$ can be obtained by deleting zero or more points from the diagram of $\pi$, \textit{i.e.}\ if $\pi$ has a subsequence whose entries have the same relative order as the entries of $\sigma$. We say $\pi$ avoids $\sigma$ if $\pi$ does not contain $\sigma$. For instance, for $\pi = 493125876$, the subsequence $9356$ is an occurrence of $\sigma = 4123$, but on the other hand $\pi$ avoids $3142$. See Figure \ref{fig:contain}.
\begin{figure}
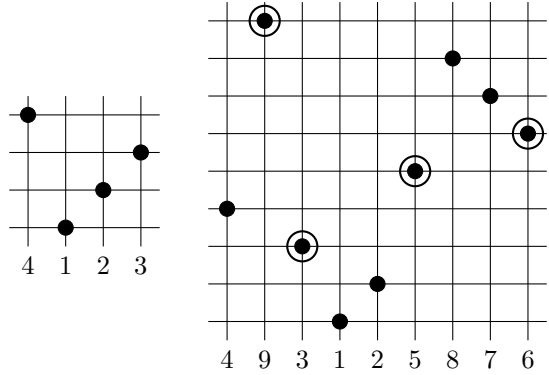

\[ \drawpermutation{4,1,2,3}{4} \hspace{0.25in}
        \drawpattern{4,9,3,1,2,5,8,7,6}{9}{9,3,5,6} \]
    \caption{The permutation $4123$ is contained in the permutation $493125876$.}
    \label{fig:contain}
\end{figure}

The set of permutations (of all sizes) is a poset under pattern containment. A \emph{permutation class} is a down-set in this poset: that is, a set $\CC$ of permutations such that, if $\pi \in \CC$ and $\sigma$ is contained in $\pi$, then $\sigma \in \CC$. For a permutation class $\CC$, we let $\CC_n$ denote the set of size-$n$ permutations in $\CC$. If $R$ is a set of permutations, then $\Av{R}$ (resp.\ $\Avn{n}{R}$) denotes the set of all (resp.\ size-$n$) permutations that avoid every element of $R$. Then $\Av{R}$ is a permutation class, and for every permutation class $\CC$ there is a unique set $R$ such that $\CC = \Av{R}$ and no element of $R$ contains another. This $R$ is called the \emph{basis} of $\CC$.

Given two permutations $\sigma$ and $\tau$ of sizes $a$ and $b$ respectively, their \emph{sum} $\sigma \oplus \tau$ is the permutation of size $a+b$ obtained by juxtaposing the diagrams of $\sigma$ and $\tau$ diagonally: that is, $(\sigma \oplus \tau)(i) = \sigma(i)$ if $1 \le i \le a$, and $(\sigma \oplus \tau)(i) = \tau(i-a)$ if $a+1 \le i \le a+b$. A class $\CC$ is \emph{sum closed} if $\sigma, \tau \in \CC$ implies $\sigma \oplus \tau \in \CC$. Given a set $A$ of permutations, the \emph{sum closure} of $A$, denoted $\textstyle\bigoplus A$, is the smallest sum closed class containing $A$. A permutation is \emph{sum-indecomposable}, or \emph{indecomposable}, if it is not the sum of two permutations of non-zero size. The set of indecomposable permutations in a class $\CC$ is denoted $\ind{\CC}$. The \emph{skew sum} of $\sigma$ and $\tau$, denoted $\sigma \ominus \tau$, is defined similarly, juxtaposing the diagrams anti-diagonally; and likewise for the notions of \emph{skew-sum closed class} and \emph{skew sum--indecomposable permutation}.

The \emph{upper growth rate} of a permutation class $\CC$, denoted $\upgr{\CC}$, is defined as $\limsup_{n\to\infty} |\CC_n|^{1/n}$. The \emph{lower growth rate}, denoted $\logr{\CC}$, is defined as $\liminf_{n\to\infty} |\CC_n|^{1/n}$. If the upper and lower growth rates of $\CC$ are equal, \textit{i.e.}\ if $\lim_{n\to\infty} |\CC_n|^{1/n}$ exists (or is $\infty$), then this number is called the \emph{proper growth rate} of $\CC$, denoted $\gr{\CC}$. We define $\upgr{\ind{\CC}}$, $\logr{\ind{\CC}}$, and $\gr{\ind{\CC}}$ similarly. More generally, for a sequence of non-negative real numbers $a_n$, the upper, lower, and proper growth rates of $a_n$ are defined the same way, respectively denoted $\upgr{a_n}$, $\logr{a_n}$, and $\gr{a_n}$.

By the Marcus--Tardos Theorem (formerly the Stanley--Wilf Conjecture), every permutation class has a finite upper growth rate except the class of all permutations \cite{MarcusTardos}. It is also known that every sum closed (or skew-sum closed) class has a proper growth rate (essentially due to Arratia \cite{Arratia}). Thus, when $\CC$ is assumed to be sum closed, we can write $\gr{\CC}$ for its proper growth rate. It is widely believed that every permutation class has a proper growth rate, but we will refer to the upper or lower growth rate unless we know for sure.

\subsection{The reverse--complement map and centrosymmetry} \label{subsec:rc}

The reverse--complement of a permutation $\pi$, denoted $\rc{\pi}$, is the permutation obtained from $\pi$ by rotating its diagram by a half turn. Equivalently, if $\pi = \pi(1)\ldots\pi(n)$, then the $i$th entry of $\rc{\pi}$ is given by $n+1-\pi(n+1-i)$. This defines a map $rc$ from the set of permutations to itself. The name comes from the fact that it is the composition of the reverse map (horizontal reflection of the diagram) and the complement map (vertical reflection of the diagram); these two maps commute.

The reverse--complement map preserves permutation containment: that is, if $\pi$ contains $\sigma$, then $\rc{\pi}$ contains $\rc{\sigma}$. Consequently, the image of a permutation class $\CC$ under $rc$ is a permutation class, denoted $\rc{\CC}$. Since $rc$ is an involution on the set of permutations, we have $|\rc{\CC_n}| = |\CC_n|$ for all $n$. A class is \emph{$rc$-invariant} if $\rc{\CC} = \CC$. A permutation $\pi$ is \emph{centrosymmetric} if $\rc{\pi} = \pi$. 

We are concerned with the number of centrosymmetric permutations in a class $\CC$. Past research has focused on finding this number for specific classes $\CC$. Egge \cite{Egge} found the number of centrosymmetric permutations in $\Avn{n}{R}$ for every set $R$ of size-$3$ permutations. Lonoff and Ostroff \cite{LO} did the same when $R$ consists of one size-$3$ and one size-$4$ permutation. Egge \cite{Egge2} found an expression for $\Avn{n}{k\ldots1}$ for arbitrary $k$, using the Robinson--Schensted algorithm and evacuation of standard Young tableaux.

\begin{table}[b]
\[ \begin{tabular}{c|c|c|l}
$\CC$ & $|\inv{\CC}_{2k}|$ & $|\inv{\CC}_{2k+1}|$ & $\left( |\inv{\CC}_0|, |\inv{\CC}_1|, \ldots \right)$ \\
\hline
$\Av{321}$ & $\textstyle\binom{2k}{k}$ & $c_k$ & $(1, 1, 2, 1, 6, 2, 20, 5, 70, 14, 252, 42, \ldots)$ \\
$\Av{321, 3412}$ & $f_{2k+2}$ & $f_{2k}$ & $(1, 1, 2, 1, 5, 2, 13, 5, 34, 13, 89, 34, \ldots)$ \\
$\Av{312, 231}$ & $2^k$ & $2^k$ & $(1, 1, 2, 2, 4, 4, 8, 8, 16, 16, 32, 32, \ldots)$ \\
$\Av{321, 312, 231}$ & $f_{k+2}$ & $f_{k+1}$ & $(1, 1, 2, 1, 3, 2, 5, 3, 8, 5, 13, 8, \ldots)$
\end{tabular} \]
\caption{$|\inv{\CC}_{2n}|$ and $|\inv{\CC}_{2n+1}|$ for various $rc$-invariant classes $\CC$, due to Egge \cite{Egge}. We let $f_n$ denote the $n$th Fibonacci number, $c_n$ the $n$th Catalan number. That $\textstyle\left|\Avn{2k}{321}\right| = \binom{2k}{k}$ was first proved by C. K. Fan and J. R. Stembridge, in the setting of fully commutative elements of type $B$; see \cite[Cor.\ 5.6 \& Prop.\ 5.9]{Stembridge}).} \label{table:egge}
\end{table}
The set of centrosymmetric permutations in $\CC$ (resp.\ $\CC_n$) is denoted $\inv{\CC}$ (resp.\ $\inv{\CC}_n$). If $n=2k+1$ is odd, then every centrosymmetric permutation $\pi$ of size $n$ must have an entry in the center column in the center row of the diagram, \textit{i.e.}\ $\pi(k+1) = k+1$. Because of this, the process of enumerating $\inv{\CC}_n$ is typically different between even and odd $n$, with $\inv{\CC}_{2k+1}$ often being obtained in a straightforward way from $\inv{\CC}_{2k}$ or $\CC_k$. As an illustration of this phenomenon, see Table \ref{table:egge}, which lists explicit formulas for $|\inv{\CC}_{2n}|$ and $|\inv{\CC}_{2n+1}|$ for various classes $\CC$, due to Egge \cite{Egge}. In this paper, then, we will be concerned almost exclusively with $\inv{\CC}_{2n}$, and we will see that it is natural to compare $\inv{\CC}_{2n}$ to $\CC_n$.

The notation we have given for growth rates of permutation classes is already established in the literature, but we now introduce analogous notions for the centrosymmetric permutations in a class. The \emph{upper $rc$--growth rate} of a permutation class $\CC$, denoted $\upgrrc{\CC}$, is defined as $\limsup_{n\to\infty} |\inv{\CC}_{2n}|^{1/n}$. The \emph{lower $rc$--growth rate}, denoted $\logrrc{\CC}$, is defined as $\liminf_{n\to\infty} |\inv{\CC}_{2n}|^{1/n}$. These are the upper and lower growth rates of the sequence $|\inv{\CC}_{2n}|$. We let $\indinv{\CC}$ denote the set of indecomposable permutations in $\CC$ that are centrosymmetric, and we define $\upgrrc{\ind{\CC}}$, $\logrrc{\ind{\CC}}$, and $\grrc{\ind{\CC}}$ similarly.

Suppose a class $\DD$ is not $rc$-invariant (meaning $\rc{\DD} \not= \DD$). The class $\DD \cap \rc{\DD}$ is $rc$-invariant, is strictly contained in $\DD$, and includes all the centrosymmetric permutations in $\DD$. Thus, it is natural to consider $\DD \cap \rc{\DD}$ instead of $\DD$, so in this paper we are chiefly concerned with classes that are $rc$-invariant.

\subsection{Conjectures and main theorems} \label{subsec:mainideas}

This paper begins to answer the following question posed by Alexander Woo at the Permutation Patterns Conference in 2016: for which $rc$-invariant permutation classes $\CC$ do we have $\upgrrc{\CC} = \upgr{\CC}$? This question is motivated by algebraic geometry. Billey and Postnikov \cite{BP}, studying the combinatorics of smooth Schubert varieties, define an algebraic notion of pattern containment in any Coxeter group, which has been called \emph{Billey--Postnikov containment} or \emph{BP containment}. BP containment in type $A$ is very similar to classical permutation pattern containment. In type $B$, which corresponds to centrosymmetric permutations, Woo has translated the algebraic definition into a purely combinatorial definition, involving a condition on the pattern occurrence's intersection with its image under $rc$ \cite[Sec.\ 2.5]{Woo}. For some sets of patterns, the number of centrosymmetric permutations BP-avoiding the set is asymptotically the same as the number without this extra condition in the definition. From this observation arises the question at hand.

\begin{table}[t]
\[ \begin{tabular}{c|c|c|l|l}
$R$ & sum closed? & $\text{gr}^{rc} = \text{gr}$? & $\gr{\Av{R}}$ & $\grrc{\Av{R}}$ \\
\hline\hline
$321$ & Yes & Yes & $4$ \hfill \cite{MacMahon} & $4$ \hfill \cite{Stembridge,Egge} \\
$4321$ & Yes & Yes & $9$ \hfill \cite{Regev} & $9$ \hfill \cite{Egge2} \\
$231,312$ & Yes & Yes & $2$ \hfill \cite{SS}  & $2$ \hfill \cite{Egge} \\
$321, 3412$ & Yes & Yes & $\smallfrac{3+\sqrt{5}}{2}$ \hfill \cite{West1996} & $\smallfrac{3+\sqrt{5}}{2}$ \hfill \cite{Egge} \\
$321, 3142$ & Yes & Yes & $\smallfrac{3+\sqrt{5}}{2}$ \hfill \cite{West1996} & $\smallfrac{3+\sqrt{5}}{2}$ \hfill \cite{LO} \\
$321, 231, 312$ & Yes & Yes & $\smallfrac{1+\sqrt{5}}{2}$ \hfill \cite{SS} & $\smallfrac{1+\sqrt{5}}{2}$  \\
$2413, 3142$ & Yes & Yes & $3+2\sqrt{2}$ \hfill \cite{West1995} & $3+2\sqrt{2}$  \\
$4321, 3412$ & Yes & Yes & $4$ \hfill \cite{KS} & $4$ \\
$4321, 3142$ & Yes & Yes & $2 + \sqrt{3}$ \hfill \cite{Vatter2012} & $2 + \sqrt{3}$ \\
\hline $321, 2143$ & No & Yes & $2$ \hfill \cite{BJS} & $2$ \hfill \cite{LO} \\
$3412, 2143$ & No & Yes & $4$ \hfill \cite{Atkinson1998} & $4$  \\
$4231, 1324$ & No & \emph{No} & $2 + \sqrt{2}$ \hfill \cite{AAV} & $2$  \\
$4321, 2143$ & No & \emph{No} & $\smallfrac{3+\sqrt{5}}{2}$ \hfill \cite{AAB} & $2$
\end{tabular} \]
\caption{Examples of $rc$-invariant classes $\Av{R}$ for which the basis $R$ consists of a few short patterns. We indicate whether the class is sum closed and whether $\grrc{\Av{R}} = \gr{\Av{R}}$, and we list the growth rate and the $rc$--growth rate. The ones that are not sum closed are also not skew-sum closed. For both classes $\CC$ for which $\grrc{\CC}$ does not equal $\gr{\CC}$, we have $\grrc{\CC} < \gr{\CC}$. For each growth rate or $rc$--growth rate given, we cite the paper that first gave the exact enumeration, from which the growth rate is easily obtained; the only exception is $\Av{4321}$, for which we cite the growth rate because it was discovered before the exact enumeration. In cases where the $rc$--growth rate does not have a citation next to it, it has not previously been computed; in these cases we obtained $\grrc{\CC}$ by finding an exact enumeration of $\inv{\CC}_{2n}$ using \textit{ad hoc} methods, typically in terms of the previously known enumeration of $\CC_n$.} \label{table:2by4}
\end{table}
We begin by looking at the $rc$--growth rates of $rc$-invariant classes whose basis consists of a few short patterns. In particular, we investigate ``$2 \times 4$ classes'', which are classes of the form $\Av{\sigma, \tau}$ for two permutations $\sigma$ and $\tau$ of size $4$. For the last several years, $2 \times 4$ classes have been a fertile testing ground for enumerative questions about permutation classes. For several $rc$-invariant classes $\CC$ for which $\gr{\CC}$ was previously known, we obtained $\grrc{\CC}$ by finding an exact enumeration of $\inv{\CC}_{2n}$ using \textit{ad hoc} methods. These results are summarized in Table \ref{table:2by4}. In most of the examples we find that $\grrc{\CC} = \gr{\CC}$, but there are two in which $\grrc{\CC} < \gr{\CC}$. Table \ref{table:2by4} justifies the choice to define the $rc$--growth rate with $\inv{\CC}_{2n}$ instead of $\inv{\CC}_n$, and it also leads to the following conjecture:

\begin{conj} \label{conj:main}
For any permutation class $\CC$, $\upgrrc{\CC} \le \upgr{\CC}$.
\end{conj}

Theorems \ref{thm:geometric} and \ref{thm:onedirection} give hypotheses under which the other direction of inequality holds. The terms used in Theorem \ref{thm:geometric} are defined in Section \ref{sec:geometric}.

\begin{thm}
\label{thm:geometric} If $\CC$ is an $rc$-invariant geometric grid class whose cell graph is a forest, then $\logrrc{\CC} \ge \gr{\CC}$.
\end{thm}

We write $\gr{\CC}$ here because, as shown by Bevan \cite{Bevan}, geometric grid classes have proper growth rates. In Section \ref{sec:geometric} we prove a stronger version of this theorem, as Theorem \ref{thm:geometric2}.

\begin{thm} \label{thm:onedirection}
If $\CC$ is sum closed and $rc$-invariant, then $\logrrc{\CC} \ge \gr{\CC}$.
\end{thm}

(This result also applies to skew-sum closed classes, as do all our results on sum closed classes, because the reverse of a centrosymmetric permutation is centrosymmetric.) We write $\gr{\CC}$ in this theorem because $\CC$, being sum closed, has a proper growth rate. In Section \ref{sec:centrosymmetric-sumclosed} we prove stronger results that imply Theorem \ref{thm:onedirection}, but we provide a quick proof now.

\begin{proof}[Proof of Theorem \ref{thm:onedirection}]
We define an injection $\CC_n \to \inv{\CC}_{2n}$ as follows: given $\sigma \in \CC_n$, define $\rho = \rc{\sigma} \oplus \sigma$. Since $\CC$ is $rc$-invariant, $\rc{\sigma} \in \CC_n$; then, since $\CC$ is sum closed, $\rho = \rc{\sigma} \oplus \sigma \in \CC_{2n}$; and
\[ \rc{\rho} = \rc{\sigma} \oplus \rc{\rc{\sigma}} = \rc{\sigma} \oplus \sigma = \rho, \]
so $\rho \in \inv{\CC}_{2n}$. This injection shows that $|\CC_n| \le |\inv{\CC}_{2n}|$, which proves the desired inequality on the growth rates.
\end{proof}

As a result of this theorem, Conjecture \ref{conj:main} would imply another conjecture:

\begin{conj} \label{conj:mainsumclosed}
If $\CC$ is sum closed and $rc$-invariant, then $\grrc{\CC}$ exists and $\grrc{\CC} = \gr{\CC}$.
\end{conj}

This conjecture is supported by the fact that, although we know several examples of $rc$-invariant $\CC$ where $\grrc{\CC} \not= \gr{\CC}$, none of these examples is sum closed. We have proved Conjecture \ref{conj:mainsumclosed} in the following special case:

\begin{thm} \label{main-sumclosed}
Let $\CC$ be a sum closed $rc$-invariant permutation class, and let $\xi \approx 2.30522$ be the unique positive root of $x^5-2x^4-x^2-x-1$ (as defined in \cite{PanVat}). If $\gr{\CC} \le \xi$, then $\grrc{\CC}$ exists and $\grrc{\CC} = \gr{\CC}$.
\end{thm}

In Sections \ref{sec:centrosymmetric-sumclosed} and \ref{sec:lessthanxi} we prove stronger results that imply Theorem \ref{main-sumclosed}.

The rest of the paper is organized as follows: Section \ref{sec:general-results} gives general results about $rc$--growth rates, including a proof that $\grrc{\Av{k\cdots 1}} = \gr{\Av{k\cdots1}}$. In Section \ref{sec:geometric} we focus on geometric grid classes, presenting examples where $\grrc{\CC} < \gr{\CC}$ and proving Theorem \ref{thm:geometric}. In Section \ref{sec:sumclosed}, we prove theorems on the $rc$--growth rates of sum closed classes, and we give results related to the work of Pantone and Vatter \cite{PanVat} on sum closed classes $\CC$  such that $\gr{\CC} \le \xi \approx 2.30522$. Section \ref{sec:thresholds} presents preliminary findings and open questions involving the threshold of unbounded indecomposables and the threshold of exponential indecomposables.

\section{General results on $rc$--growth rates} \label{sec:general-results}

\subsection{Basic facts} \label{sec:2by4}

\begin{prop} \label{prop:2factor}
(a) If $\CC$ is any class, then $|\inv{\CC}_{2n}| \le |\CC_{2n}|$, and so $\upgrrc{\CC} \le \left(\upgr{\CC}\right)^2$ and $\logrrc{\CC} \le \left(\logr{\CC}\right)^2$. (b) If $\CC$ is any class, then $|\inv{\CC}_{2n}| \le 2^n\,|\CC_n|$, and so $\upgrrc{\CC} \le 2\, \upgr{\CC}$ and $\logrrc{\CC} \le 2\, \logr{\CC}$.
\end{prop}

\begin{proof}
Part (a) holds because $\inv{\CC}_{2n} \subseteq \CC_{2n}$. To prove part (b), let $\rho \in \inv{\CC}_{2n}$, and let $J$ be the set of elements of $[2n]$ that occur in the first $n$ entries of $\rho$. Because $\rho$ is centrosymmetric, $j \in J$ if and only if $n+1-j \not\in J$, so there are $2^n$ possible sets $J$. Now let $\pi$ be the permutation formed by the first $n$ entries of $\rho$; then $\pi \in \CC_n$. Thus, for each $\rho \in \inv{\CC}_{2n}$, we obtain a set $J$ and a permutation $\pi \in \CC_n$, and the number of such pairs $(J, \pi)$ is $2^n\,|\CC_n|$. Moreover, the function $\rho \mapsto (J, \pi)$ just described is injective. Therefore, $|\inv{\CC}_{2n}| \le 2^n\,|\CC_n|$.

The corresponding statements about the growth rates now follow immediately.
\end{proof}

Note that the function defined in this proof is not necessarily surjective, because there may be pairs $(J, \pi)$ whose corresponding $\rho$ is not in $\CC$ even though $\pi \in \CC$. For instance, if $\CC = \Av{21}$ and $(J, \pi) = (\{3,4\}, 12)$, then $\pi = 12 \in \Av{21}$ but $\rho = 3412 \not\in \Av{21}$.

\begin{cor}
If $\upgr{\CC}$ is $0$ or $1$, then $\gr{\CC}$ and $\grrc{\CC}$ exist and $\grrc{\CC} = \gr{\CC}$.
\end{cor}

\begin{proof}
Assume $\upgr{\CC}$ is $0$ or $1$. Then clearly $\gr{\CC}$ exists. If $\gr{\CC} = 0$, then by Proposition \ref{prop:2factor}(a) we have $\logrrc{\CC} = \upgrrc{\CC} = 0$, and $\grrc{\CC}$ exists.

If $\gr{\CC} = 1$, then $|\CC_n| \ge 1$ for all $n$, so by the Erd\H{o}s--Szekeres Theorem $\CC$ includes the permutation $1\ldots n$ for all $n$ or the permutation $n \ldots 1$ for all $n$; these permutations are centrosymmetric, so $|\inv{\CC}_{2n}| \ge 1$ for all $n$, and so $\logrrc{\CC} \ge 1$. But by Proposition \ref{prop:2factor}(a) we have $\upgrrc{\CC} \le 1$, so in fact $\logrrc{\CC} = \upgrrc{\CC} = 1$, and $\grrc{\CC}$ exists.
\end{proof}

\subsection{Unions of permutation classes} \label{sec:unions}

Let $\DD$ be a class, and let $\CC = \DD \cup \rc{\DD}$. Then $\CC$ and $\DD \cap \rc{\DD}$ are both $rc$-invariant, and $\inv{\CC} = \inv{(\DD \cap \rc{\DD})}$. But if $\DD$ is a proper subclass of $\CC$, then $\DD \cap \rc{\DD}$ is also a proper subclass of $\CC$; in this case, we should expect $\inv{\CC}$ to grow slowly relative to $\CC$, because all the centrosymmetric permutations in $\CC$ are confined to the smaller class $\DD \cap \rc{\DD}$. This expectation is consistent with the fact that, of the three classes $\CC$ of this form that we have checked, all of them satisfy $\grrc{\CC} < \gr{\CC}$, as seen in Table \ref{table:3examples}.

\begin{table}
\[ \begin{tabular}{c|cl|cl}
$\DD$ & $\gr{\CC}$ & cite & $\grrc{\CC}$ [$= \grrc{\DD \cap \rc{\DD}}$] & cite (if previously known) \\
\hline
$\Av{312}$ & $4$ & \cite{Knuth} & $2$ & \cite{Egge} \\
$\Av{4123}$ & $9$ & \cite{Stankova} & $4$ & \\
$\Av{4312}$ & $9$ & \cite{West} & $2+\sqrt{5}$ &
\end{tabular} \]
\caption{Three examples of classes of the form $\CC = \DD \cup \rc{\DD}$ and their $rc$--growth rates. The growth rate of $\CC$ equals the growth rate of $\DD$, which was already known in these examples. The $rc$--growth rate of $\CC$ equals the $rc$--growth rate of $\DD \cap \rc{\DD}$, which we computed using \textit{ad hoc} methods if it was not already known. } \label{table:3examples}
\end{table}

Thus it makes sense to focus our investigation on classes $\CC$ that cannot be written as $\DD \cup \rc{\DD}$ unless $\DD = \CC$, and this motivates a definition:

\begin{defn}
For an $rc$-invariant class $\CC$, let $\CChat$ denote the intersection of all classes $\DD$ such that $\CC = \DD \cup \rc{\DD}$. If $\CChat = \CC$, meaning that $\CC$ cannot be written as $\DD \cup \rc{\DD}$ unless $\DD = \CC$, then we say that $\CC$ is \emph{$rc$-atomic}.
\end{defn}

A class is called \emph{atomic} if it cannot be written as a union of two proper subclasses, so every $rc$-invariant atomic class is $rc$-atomic. Also note that $\CChat$, as an intersection of classes, is a class.

\begin{prop} \label{prop:characterization}
Let $\CC$ be $rc$-invariant.
\begin{itemize}
\item[(a)] Let $\sigma \in \CC$, and for any permutation $\alpha$ let $\CC(\alpha)$ denote the class of permutations in $\CC$ that avoid $\alpha$. The following are equivalent:
\begin{itemize}
\item[(i)] $\sigma$ is in every $\DD$ such that $\CC = \DD \cup \rc{\DD}$ --- that is, $\sigma \in \CChat$;
\item[(ii)] $\CC(\sigma) \cup \CC(\rc{\sigma}) \not= \CC$;
\item[(iii)] There is $\pi \in \CC$ that contains $\sigma$ and $\rc{\sigma}$.\end{itemize}
\item[(b)] $\CC$ is $rc$-atomic if and only if for every $\sigma \in \CC$ there is $\pi \in \CC$ that contains $\sigma$ and $\rc{\sigma}$.
\item[(c)] The centrosymmetric permutations in $\CC$ all lie in $\CChat$ --- that is, $\inv{\CC} = \inv{\CChat}$.
\end{itemize}
\end{prop}

\begin{proof}
$\text{(i)} \Rightarrow \text{(ii)}$: Suppose (ii) is false, so $\CC(\sigma) \cup \CC(\rc{\sigma}) = \CC$. Since $\sigma \not\in \CC(\sigma)$, this contradicts (i).

$\text{(ii)} \Rightarrow \text{(iii)}$: By (ii), there is $\pi \in \CC$ that is not in $\CC(\sigma) \cup \CC(\rc{\sigma})$. Then $\pi$ avoids neither $\sigma$ nor $\rc{\sigma}$, which implies (iii).

$\text{(iii)} \Rightarrow \text{(i)}$: Let $\pi \in \CC$ contain $\sigma$ and $\rc{\sigma}$. If $\DD$ is a class such that $\CC = \DD \cup \rc{\DD}$, then either $\pi \in \DD$ or $\pi \in \rc{\DD}$; in the former case we get $\sigma \in \DD$, and in the latter case we get $\rc{\sigma} \in \rc{\DD}$ so $\sigma \in \DD$. Thus $\sigma$ is in every $\DD$ such that $\CC = \DD \cup \rc{\DD}$, which means that $\sigma \in \CChat$.

Part (b) follows immediately from the equivalence of conditions (i) and (iii). For part (c), let $\rho \in \inv{\CC}$, meaning $\rho = \rc{\rho}$. If $\CC = \DD \cup \rc{\DD}$, then without loss of generality $\rho \in \DD$, so $\rho = \rc{\rho} \in \rc{\DD}$, and so $\rho \in \DD \cap \rc{\DD}$. Therefore $\rho \in \CChat$.
\end{proof}

The property in (b) is an analog of the joint-embedding property, which a class $\CC$ satisfies when for every $\sigma, \tau \in \CC$ there is $\pi \in \CC$ that contains $\sigma$ and $\tau$. The joint-embedding property is equivalent to being atomic.

Let $\CC$ be $rc$-invariant. As we discussed above, we should not expect $\upgrrc{\CC} = \upgr{\CC}$ if $\CC$ is not $rc$-atomic. We could hope that this equality must hold when $\CC$ is $rc$-atomic, or under either of two stronger conditions: that $\CC$ is atomic, or that $\CC$ is generated by the permutations in $\textstyle\bigcup_n \CC_{2n}^{rc}$ (the even-size centrosymmetric permutations in $\CC$). We will see in Proposition \ref{prop:counterexample} that even these strong conditions are not enough.

\subsection{Centrosymmetric permutations avoiding a monotone pattern}

In this subsection we prove:

\begin{thm} \label{thm:monotone}
For all $k \ge 1$, $\grrc{\Av{k \ldots 1}}$ exists and equals $\gr{\Av{k \ldots 1}}$.
\end{thm}

We remark that $\Av{k \ldots1}$ and $\Av{1\ldots k}$ are reverses of each other, so $|\Avn{n}{k \ldots 1}| = |\Avn{n}{1 \ldots k}|$ and $|\inv{\Avn{n}{k \ldots 1}}| = |\inv{\Avn{n}{1\ldots k}}|$; thus this theorem applies to $\Av{1\ldots k}$ as well. We also remark that $\gr{\Av{k \ldots 1}} = (k-1)^2$, as proved by Regev \cite{Regev}. 

\begin{proof}[Proof of Theorem \ref{thm:monotone}]
Theorem \ref{thm:onedirection} shows that $\logrrc{\Av{k\ldots1}} \ge \gr{\Av{k \ldots 1}}$, so now it suffices to show that $\upgrrc{\Av{k\ldots1}} \le \gr{\Av{k \ldots 1}}$.

Set $a_m^j = |\Avn{m}{j\ldots1}|$. Egge \cite{Egge2} proved that
\begin{equation} \label{eqn:egge}
|\inv{\Avn{2n}{k \ldots 1}}| = \sum_{i=0}^n \binom{n}{i}^2 a_i^{\lceil (k+1)/2\rceil} \, a_{n-i}^{\lfloor (k+1)/2 \rfloor}.
\end{equation}
Set $p = \lceil (k+1)/2 \rceil$ and $q = \lfloor (k+1)/2 \rfloor$. Let $\varepsilon > 0$. Because $\gr{ a_m^p} = (p-1)^2$ and $\gr{ a_m^q} = (q-1)^2$, there is a constant $t$ such that $a_m^p \le t(1+\varepsilon)^m (p-1)^{2m}$ and $a_m^q \le t(1+\varepsilon)^m(q-1)^{2m}$ for all $m$.

We now employ a trick used by Claesson, Jel\'inek, and Steingr\'imsson \cite[Lem.\ 4]{CJS}:
\begin{align*}
|\inv{\Avn{2n}{k \ldots 1}}| &\le t^2 (1+\varepsilon)^n \sum_{i=0}^n \binom{n}{i}^2 (p-1)^{2i} (q-1)^{2(n-i)} & \text{(by \eqref{eqn:egge})} \\
&= t^2(1+\varepsilon)^n \sum_{i=0}^n \left[ \binom{n}{i} (p-1)^i (q-1)^{n-i} \right]^2 \\
&\le t^2(1+\varepsilon)^n \left[ \sum_{i=0}^n \binom{n}{i} (p-1)^i (q-1)^{n-i} \right]^2 & \text{(because $\textstyle\sum_i |x_i|^2 \le \left(\sum_i |x_i| \right)^2$)} \\
&= t^2(1+\varepsilon)^n (p+q-2)^{2n} \\
&= t^2 (1+\varepsilon)^n (k-1)^{2n}.
\end{align*}
The quantity $t^2 (1+\varepsilon)^n (k-1)^{2n}$ has growth rate $(1+\varepsilon)(k-1)^2$, so we have shown that $\upgrrc{\Av{k\ldots1}} \le (1+\varepsilon)(k-1)^2$. This holds for all $\varepsilon > 0$, so $\upgrrc{\Av{k\ldots1}} \le (k-1)^2$.
\end{proof}

\section{Geometric grid classes} \label{sec:geometric}

Let $A$ be a $\{0,1,-1\}$-matrix. The \emph{standard figure} of $A$ is obtained by replacing each $1$ (resp.\ $-1$) in $A$ with a line segment of slope $1$ (resp.\ $-1$) and replacing each $0$ with empty space. For instance, $\begin{pmatrix}1 & -1 \\ 1 & 0\end{pmatrix}$ has
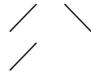
\[ \begin{matrix} \diagup & \diagdown \\ \diagup & \end{matrix} \]
as its standard figure. If we choose $n$ points on the standard figure of $A$ such that no two have the same horizontal or vertical coordinate, the result is a permutation $\pi$ of size $n$, and the set of points is called a \emph{drawing} of $\pi$ (on $A$). The set of permutations obtained in this way is a permutation class called the \emph{geometric grid class} of $A$, denoted $\geom{A}$. For instance, $\Geom{\begin{matrix}1 & -1 \end{matrix}}$ is the class of permutations made of an increasing sequence followed by a decreasing sequence. Also let $\geomn{A}{n}$ denote the set of size-$n$ permutations in $\geom{A}$.

Geometric grid classes were studied in depth in \cite{AABRV}; in particular, it is shown that $\geom{A}$ is atomic and has a rational generating function. Bevan \cite{Bevan} shows that $\geom{A}$ has a proper growth rate and gives a way to find that growth rate from $A$. In an abuse of notation, we will refer to $rc$ acting on the entries of a centrosymmetric permutation $\pi$, the cells of a centrosymmetric matrix $A$, or the points in a drawing of $\pi$ on $A$.

\begin{prop}
If $A$ is a centrosymmetric $\{0,1,-1\}$-matrix, then $\geom{A}$ is $rc$-invariant, and $\geom{A}$ is generated by the permutations in $\textstyle\bigcup_n \inv{\geomn{A}{2n}}$ (the even-size centrosymmetric permutations in $\geom{A}$).
\end{prop}

\begin{proof}
Let $\pi \in \geomn{A}{n}$. Since $A$ is centrosymmetric, applying $rc$ to a drawing of $\pi$ on $A$ results in a drawing of $\rc{\pi}$ on $A$, proving that $\geom{A}$ is $rc$-invariant. Furthermore, the union of these drawings of $\pi$ and $\rc{\pi}$ is a centrosymmetric set of points, which, after perturbing any points with the same horizontal or vertical coordinate, is the drawing of a centrosymmetric permutation $\rho$. We have $\rho \in \inv{\geomn{A}{2n}}$, and $\pi$ is contained in $\rho$.
\end{proof}

We now come to another example where $\grrc{\CC} < \gr{\CC}$: namely, $\CC = \Geom{\begin{matrix} -1 & 1 \\ 1 & -1 \end{matrix}}$. The standard figure of this matrix is an $X$, and this class has been called the \emph{$X$-class}. It has been enumerated by Elizalde \cite{ElizaldeX}, and its growth rate is $2+\sqrt{2}$. However:

\begin{prop} \label{prop:counterexample}
For $\CC = \Geom{\begin{matrix} -1 & 1 \\ 1 & -1 \end{matrix}}$, we have $|\inv{\CC}_{2n}| = 2^n$.
\end{prop}

\begin{proof}
Let $\pi \in \inv{\CC}_{2n}$ for $n \ge 1$. By \cite[Lem.\ 3.1]{ElizaldeX}, $\pi$ must have an entry in at least one of the four corners --- that is, $\pi(1) \in \{1, 2n\}$ or $\pi(2n) \in \{1, 2n\}$. Since $\pi$ is centrosymmetric, it must have an entry in two opposite corners --- that is, $\{\pi(1), \pi(2n)\} = \{1, 2n\}$. This gives us a total of two options for $\pi(1)$ and $\pi(2n)$; removing these entries yields a permutation in $\inv{\CC}_{2n-2}$, and the result follows by induction.
\end{proof}

Thus, $\grrc{\CC} = 2 < \gr{\CC}$. This example is dramatic: even for a class that is atomic, is generated by its centrosymmetric permutations, and has a rational generating function, it is not necessarily true that $\upgrrc{\CC} = \upgr{\CC}$.

Let $A$ be a $\{0,1,-1\}$-matrix. An \emph{$A$-gridded} permutation (on $A$) is a permutation $\pi$ with a valid choice of which cell of $A$ to draw each entry of $\pi$ on. Let $\griddings{A}$ be the set of $A$-gridded permutations. Given $A$, each $\pi$ has a finite number of griddings on $A$, and the maximum number of griddings over all size-$n$ permutations is bounded above by a polynomial in $n$; thus $\gr{\griddings{A}} = \gr{\geom{A}}$.

Again abusing notation, we say $rc$ acts on $A$-gridded permutations (when $A$ is centrosymmetric), and we say an $A$-gridded permutation fixed by $rc$ is centrosymmetric. In order for a gridded permutation to be centrosymmetric, the permutation must be centrosymmetric and its gridding on $A$ must be centrosymmetric. Let $\inv{\griddings{A}}$ denote the set of centrosymmetric $A$-gridded permutations, and define $\grrc{\griddings{A}}$ the same way as the $rc$--growth rate of a permutation class.

The \emph{cell graph} of $A$ is the graph whose vertices are the non-zero cells of $A$, where two cells are adjacent if (1) they share a row or column and (2) there are no non-zero cells between them in their row or column. For instance, $\begin{pmatrix}-1 & 1 \\ 1 & -1\end{pmatrix}$ (the matrix for the $X$-class) has
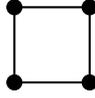
\[ \begin{tikzpicture}
\draw (0,0) [fill=black] circle (0.1);
\draw (0,1) [fill=black] circle (0.1);
\draw (1,0) [fill=black] circle (0.1);
\draw (1,1) [fill=black] circle (0.1);
\draw [thick] (0,0) -- (0,1) -- (1,1) -- (1,0) -- (0,0);
\end{tikzpicture} \]
as its cell graph. The fact that this is a cycle will help explain the $X$-class's behavior, as we will see in Theorem \ref{thm:geometric2}.

If $A$ is centrosymmetric, then $rc$ acting on the cells of $A$ induces an automorphism of the cell graph of $A$. Again abusing notation, we will call this automorphism $rc$. In particular, $rc$ maps each component of the graph onto either itself or a different component.

\begin{thm} \label{thm:geometric2}
Let $A$ be a centrosymmetric $\{0,1,-1\}$-matrix, let $G$ be the cell graph of $A$, and assume without loss of generality that $A$ has an even number of rows and an even number of columns. Each statement implies the next:
\begin{itemize}
\item[(i)] $G$ is a forest (has no cycles);
\item[(ii)] $rc$ maps every component of $G$ onto a different component;
\item[(iii)] $\logrrc{\geom{A}} \ge \gr{\geom{A}}$.
\end{itemize}
\end{thm}

We remark that, if $A$ has an odd number of rows or an odd number of columns, then $A$ can be replaced with the matrix $A^{\times2}$ obtained by replacing each $1$ with $\begin{pmatrix}0 & 1 \\ 1 & 0\end{pmatrix}$, each $-1$ with $\begin{pmatrix} -1 & 0 \\ 0 & -1\end{pmatrix}$, and each $0$ with $\begin{pmatrix} 0 & 0 \\ 0 & 0 \end{pmatrix}$. The standard figure of $A^{\times2}$ is the same as that of $A$, just stretched by a factor of $2$ in each direction. Consequently, $\geom{A^{\times2}} = \geom{A}$, which is why there is no loss of generality from the assumption that $A$ has an even number of rows and an even number of columns.

\begin{proof}[Proof of Theorem \ref{thm:geometric2}]
$\text{(i)} \Rightarrow \text{(ii)}$: Assume $G$ is a forest, and suppose $G$ has a component that is mapped onto itself by $rc$. This component must be a tree; call this tree $T$. Let $v$ be a vertex in $T$; since $rc$ maps $T$ to itself, $rc(v)$ is also in $T$. Thus there is a path in $T$ between $v$ and $rc(v)$; call this path $P$. Observe that $rc(P)$ is also a path in $T$ between $v$ and $rc(v)$, but there is only one such path because $T$ is a tree, so $rc(P) = P$. Thus the center element of $P$, which is a vertex or edge of $G$, is mapped to itself by $rc$. But $G$ cannot have a vertex or edge mapped to itself by $rc$, because $A$ has an even number of rows and an even number of columns. This is a contradiction, so no component of $G$ is mapped onto itself.

$\text{(ii)} \Rightarrow \text{(iii)}$: Assume $rc$ maps every component of $G$ onto a different component. Thus the components of $G$ come in pairs, each pair consisting of two components that map onto each other under $rc$. Let $X$ be a subgraph consisting of one component from each pair, and let $Y$ be the subgraph consisting of the other components. Then $X$ and $Y$ form a partition of the vertices and edges of $G$, and $Y = \rc{X}$, and there are no edges between $X$ and $Y$.

Let $A_X$ (resp.\ $A_Y$) be the matrix obtained from $A$ by keeping the cells that are vertices in $X$ (resp.\ $Y$) and replacing the rest of the cells with $0$. Note that $A_Y = \rc{A_X}$, so $|\geomn{A_X}{n}| = |\geomn{A_Y}{n}|$ and $|\griddingsn{A_X}{n}| = |\griddingsn{A_Y}{n}|$. Because there are no edges between $X$ and $Y$ in the cell graph, there is no non-zero entry of $A_X$ in the same row or column as a non-zero entry of $A_Y$.

Recall that $\gr{\geom{A}} = \gr{\griddings{A}}$. Every $A$-gridded permutation is obtained from a pair of an $A_X$-gridded permutation and an $A_Y$-gridded permutation. Every such pair of gridded permutations gives rise to exactly one $A$-gridded permutation, because points placed on $A_X$ and points placed on $A_Y$ do not interleave in multiple ways. Therefore, $\left|\griddingsn{A}{n}\right| = \sum_{k=0}^n {\left|\griddingsn{A_X}{k}\right|}^2$, and thus $\gr{\griddings{A}} = \gr{\griddings{A_X}}$.

Moreover, we have a bijection between $\griddingsn{A_X}{n}$ and $\inv{\griddingsn{A}{2n}}$: given an $A_X$-gridded permutation $\pi$, take the union of the drawing of $\pi$ on $A_X$ and the drawing of $\rc{\pi}$ on $A_Y$, yielding a centrosymmetric $A$-gridded permutation. This is a bijection because, again, points placed on $A_X$ and points placed on $A_Y$ do not interleave in multiple ways. Therefore, $\left|\griddingsn{A_X}{n}\right| = \left|\inv{\griddingsn{A}{2n}}\right|$, and in particular $\gr{\griddings{A_X}} = \grrc{\griddings{A}}$ (and the latter growth rate is proper).

Finally, because the maximum number of griddings of a permutation of size $2n$ is bounded above by a polynomial, $\grrc{\griddings{A}}$ equals the growth rate of the number of size-$2n$ centrosymmetric permutations in $\geom{A}$ \emph{that have a centrosymmetric gridding on $A$}, which is less than or equal to $\logrrc{\geom{A}}$.
\end{proof}

The reason we get an inequality instead of an equality at the end of this proof is subtle: a centrosymmetric permutation in $\geom{A}$ by definition can be drawn on the standard figure of $A$, but not necessarily in a centrosymmetric way. The smallest instance of this phenomenon is with $A = \begin{pmatrix} 1 & 0 \\ 0 & 1 \end{pmatrix}$. The permutation $12$ is centrosymmetric, and it can be drawn on the standard figure of $A$, but every drawing of it has both entries of $12$ in the top-left cell or both entries in the lower-right cell, neither of which is a centrosymmetric gridding. More complicated instances of this phenomenon are not hard to find.

For any centrosymmetric $\{0,1,-1\}$-matrix $A$, we conjecture that ``almost all'' even-size centrosymmetric permutations in $\geom{A}$ have a centrosymmetric gridding on $A$, in the sense that the ones without such a gridding have a strictly smaller growth rate. This would imply that $\gr{\geom{A}} = \grrc{\geom{A}}$ under the conditions of Theorem \ref{thm:geometric2}. This equality is also implied by Conjecture \ref{conj:main}.

\section{Sum closed classes} \label{sec:sumclosed}

For the rest of this paper, we assume $\CC$ is a sum closed class. In this section, we investigate the centrosymmetric permutations in a sum closed class (Section \ref{sec:centrosymmetric-sumclosed}), and we give results drawing from the work of Pantone and Vatter \cite{PanVat} on sum closed classes with growth rate $\le \xi$ (Section \ref{sec:lessthanxi}). The main result is Theorem \ref{main-sumclosed2}, which gives a few conditions that each individually imply that $\grrc{\CC}$ exists and equals $\gr{\CC}$.

Recall that, if $A(x) = \sum_{n\ge0} |\CC_n|\,x^n$ and $C(x) = \sum_{n\ge1} |\ind{\CC}_n|\,x^n$, then
\begin{equation}
A(x) = \frac{1}{1-C(x)}. \label{eqn:ac}
\end{equation}
Also, $\CC$ has a proper growth rate, a fact which has been known essentially since Arratia \cite{Arratia}.

\subsection{Centrosymmetric permutations in a sum closed class} \label{sec:centrosymmetric-sumclosed}

For this subsection we assume that $\CC$ is $rc$-invariant. Recall from Section \ref{sec:unions} that $\CC$ is $rc$-atomic if it is not of the form $\DD \cup \rc{\DD}$ unless $\DD = \CC$, and $\CC$ is atomic if it is not the union of two proper subclasses. Since we assume $\CC$ is sum closed, it follows that $\CC$ is atomic and hence $rc$-atomic. The next fact we prove involves a stronger property than being $rc$-atomic.

\begin{prop}
If $\CC$ is sum closed and $rc$-invariant, then $\CC$ is generated by the permutations in ${\textstyle\bigcup_n} \,\inv{\CC}_{2n}$ (\textit{i.e.}\ the even-size centrosymmetric permutations in $\CC$).
\end{prop}

\begin{proof}
Let $\pi \in \CC$. Since $\CC$ is $rc$-invariant, $\rc{\pi} \in \CC$; since $\CC$ is sum closed, $\pi \oplus \rc{\pi} \in \CC$. Thus every $\pi \in \CC$ is contained in an even-length centrosymmetric permutation in $\CC$ (namely $\pi \oplus \rc{\pi}$), meaning $\CC$ is generated by its even-length centrosymmetric elements.
\end{proof}

Define
\begin{align*}
a_n &= |\CC_n|; & b_n &= |\inv{\CC}_{2n}|; & d_n &= |\indinv{\CC}_{2n}|; \\
A(x) &= \sum_{n\ge0} a_n x^n; & B(x) &= \sum_{n\ge0} b_n x^n; & D(x) &= \sum_{n\ge1} d_n x^n.
\end{align*}
Note that, in $B(x)$ and $D(x)$, we are taking the permutations in $\inv{\CC}_{2n}$ to have weight $n$, despite having size $2n$ as a permutation.

\begin{prop} \label{prop:gf} $B(x) = (1 + D(x))\,A(x)$.
\end{prop}

\begin{proof}
The left side counts even-size permutations in $\inv{\CC}$ (with weight half their size), and the right side counts ordered pairs $(\widetilde{\rho}, \pi)$ where $\pi \in \CC$ and $\widetilde{\rho} \in \{\varepsilon\} \cup \indinv{\CC}$ (with $\widetilde{\rho}$ of even size, counted with weight half its size, and $\varepsilon$ denotes the empty permutation). With $\pi$ and $\widetilde{\rho}$ as such, consider the permutation $\rho = \rc{\pi} \oplus \widetilde{\rho} \oplus \pi$. We see that $\rho \in \CC$ because $\rc{\pi}, \widetilde{\rho}, \pi \in \CC$; we see that $\rho$ is centrosymmetric because $\widetilde{\rho}$ is centrosymmetric, so
\[ \rc{\rho} = \rc{\pi} \oplus\rc{\widetilde{\rho}} \oplus \rc{\rc{\pi}} = \rc{\pi} \oplus \widetilde{\rho} \oplus \pi = \rho; \]
and we see that the weight of $\rho$ (which is half its size) is the sum of the weights of $\pi$ and $\widetilde{\rho}$.

Furthermore, this correspondence is a bijection. Let $\rho \in \inv{\CC}_{2n}$ be arbitrary. If $\rho$ has an even number of indecomposable blocks, then $\rho$ has a unique decomposition as $\rho = \rc{\pi} \oplus \pi$ for some $\pi \in \CC$. If $\rho$ has an odd number of indecomposable blocks, then $\rho$ has a unique decomposition as $\rho = \rc{\pi} \oplus \widetilde{\rho} \oplus \pi$ for some $\pi \in \CC$ and $\widetilde{\rho} \in \indinv{\CC}$.\end{proof}

\begin{prop} \label{inequality}
If $\CC$ is sum closed and $rc$-invariant, then $\upgrrc{\CC} = \max\{\gr{\CC}, \upgrrc{\ind{\CC}}\}$ and $\logrrc{\CC} \ge \max\{\gr{\CC}, \logrrc{\ind{\CC}}\}$.
\end{prop}

\begin{proof} From Proposition \ref{prop:gf} we obtain
\begin{equation} \label{eqn:abd} b_n = a_n + \sum_{k=1}^n a_{n-k} d_k. \end{equation}
Since all the numbers appearing in \eqref{eqn:abd} are non-negative (and $a_0 = 1$), we see that $b_n \ge d_n$ and $b_n \ge a_n$, which shows that $\upgr{b_n} \ge \max\{\gr{a_n}, \upgr{d_n}\}$ and $\logr{b_n} \ge \max\{\gr{a_n}, \logr{d_n}\}$

Let $x > \upgr{a_n}$ and $y > \upgr{d_n}$, and let $M = \max\{x,y\}$; then there is a constant $t$ such that $a_n \le tx^n$ and $d_n \le t y^n$ for all $n$. Then, by (1),
\[ b_n \le t^2 \sum_{k=0}^n x^{n-k} y^k \le t^2 (n+1) M^n. \]
The quantity $t^2(n+1)M^n$ has growth rate $M$, so we have shown that $\upgr{b_n} \le M$. This holds for every $M > \max\{\upgr{a_n}, \upgr{d_n}\}$, so we conclude that $\upgr{b_n} \le \max\{\upgr{a_n}, \upgr{d_n}\}$.
\end{proof}

Theorem \ref{thm:onedirection} follows immediately from Proposition \ref{inequality}, and we restate it here:

\begin{restate}{thm:onedirection}
If $\CC$ is sum closed and $rc$-invariant, then $\logrrc{\CC} \ge \gr{\CC}$.
\end{restate}

This theorem, if Conjecture \ref{conj:main} is true, would imply that $\grrc{\CC}$ exists and equals $\gr{\CC}$ for all sum closed $rc$-invariant $\CC$ (Conjecture \ref{conj:mainsumclosed}).

We have another corollary that follows from Proposition \ref{inequality}:

\begin{cor} \label{cor:upgrrcind}
Let $\CC$ be sum closed and $rc$-invariant. (a) $\upgrrc{\CC} = \gr{\CC}$ if and only if $\upgrrc{\ind{\CC}} \le \gr{\CC}$. (b) If the equivalent statements in (a) are true, then $\grrc{\CC}$ exists.
\end{cor}

\begin{proof}
Part (a) is immediate from the fact that $\upgrrc{\CC} = \max\{\gr{\CC}, \upgrrc{\ind{\CC}}\}$ (Proposition \ref{inequality}). For (b), if $\upgrrc{\ind{\CC}} \le \gr{\CC}$, then Proposition \ref{inequality} shows that
\[ \upgrrc{\CC} = \max\{\gr{\CC}, \upgrrc{\ind{\CC}}\} = \gr{\CC} = \max\{\gr{\CC}, \logrrc{\ind{\CC}}\} \le \logrrc{\CC}, \]
so $\upgrrc{\CC} = \logrrc{\CC}$.
\end{proof}

% The next theorem, a more detailed version of our Theorem \ref{main-sumclosed}, is the main theorem of this section. The main idea is that, if a sum closed, $rc$-invariant class $\CC$ is small enough (for various definitions of ``small''), then $\grrc{\CC}$ exists and equals $\gr{\CC}$.

We come to the main theorem of this section. The main idea is that, if a sum closed, $rc$-invariant class $\CC$ has a small number of indecomposables (for various definitions of ``small''), then $\grrc{\CC}$ exists and equals $\gr{\CC}$.

%\begin{thm} \label{main-sumclosed2}
%Let $\CC$ be sum closed and $rc$-invariant. Each statement implies the next:
%\begin{itemize}
%\item[(i)] $\gr{\CC} \le \xi$;
%\item[(ii)] $|\ind{\CC}_n|$ is bounded;
%\item[(iii)] $\gr{\ind{\CC}}$ is $0$ or $1$;
%\item[(iv)] $\grrc{\CC}$ exists and $\grrc{\CC} = \gr{\CC}$.
%\end{itemize}
%\end{thm}

\begin{thm} \label{main-sumclosed2}
Let $\CC$ be sum closed and $rc$-invariant. Each statement implies the next:
\begin{itemize}
\item[(i)] $|\ind{\CC}_n|$ is bounded;
\item[(ii)] $\gr{\ind{\CC}}$ is $0$ or $1$;
\item[(iii)] $\grrc{\CC}$ exists and $\grrc{\CC} = \gr{\CC}$.
\end{itemize}
\end{thm}

\begin{proof}
% The implication $\text{(i)} \Rightarrow \text{(ii)}$ is from Proposition \ref{prop:bounded}.
$\text{(i)} \Rightarrow \text{(ii)}$: If there is $N$ such that $|\CC_n| = 0$ for all $n \ge N$, then $\gr{\ind{\CC}} = 0$. Otherwise, $\gr{\ind{\CC}} = 1$. (It is not hard to show that $\ind{\CC}$ always has a proper growth rate in this case.)

$\text{(ii)} \Rightarrow \text{(iii)}$: By Proposition \ref{prop:2factor}(a) we have $\upgrrc{\ind{\CC}} \le \left(\upgr{\ind{\CC}}\right)^2 \le 1 \le \gr{\CC}$ (assuming $\CC$ is an infinite class); hence $\upgrrc{\ind{\CC}} \le \gr{\CC}$, and (iii) now follows from Corollary \ref{cor:upgrrcind}.
\end{proof}

Statements (i) and (ii) in Theorem \ref{main-sumclosed2} are of interest because they are weaker conditions under which $\upgrrc{\CC} = \gr{\CC}$, but we are also interested in them independently of our inquiry into centrosymmetric permutations. We are in the process of finding the highest threshold of $\gr{\CC}$ (for sum closed $\CC$) below which (i) and (ii) must hold. Theorem \ref{prop:bounded}, in the following subsection, says that these thresholds are at least $\xi$. We discuss this more in Section \ref{sec:thresholds}.

\subsection{Sum closed classes with growth rate $\le \xi$} \label{sec:lessthanxi}

Let $\xi \approx 2.30522$ be the unique positive root of $x^5-2x^4-x^2-x-1$, as defined by Pantone and Vatter \cite{PanVat}. The following result is implicit in the work of \cite{PanVat}:
\begin{prop}[see {\cite[Sec.\ 5 \& 6]{PanVat}}] \label{prop:bounded}
If $\CC$ is sum closed and $\gr{\CC} \le \xi$, then $|\ind{\CC}_n|$ (for $n \ge 1$) is bounded above by the sequence $(1, 1, 3, 5, 5, 5, 4^\infty)$ or the sequence $(1, 1, 2, 3, 4^{2i}, 5, 4^\infty)$ for some $i \ge 0$. In particular, $|\ind{\CC}_n| \le 5$ for all $n$.
\end{prop}

\begin{proof}[Proof sketch of Proposition \ref{prop:bounded}]
Let $c_n = |\ind{\CC}_n|$. Propositions 5.1, 5.3, and 5.5 of \cite{PanVat} imply this fact: if there is $k$ such that $c_k$ is below the $k$th entry of $(1, 1, 2, 3, 4^\infty)$, then $c_n$ is weakly decreasing over all $n \ge k$. Consequently, if $c_n$ is bounded above by one of the two claimed upper-bound sequences for $n < k$, then $c_n$ will continue to be bounded above by that sequence for all $n$.

As shown in \cite[Sec.\ 6]{PanVat}, we must have $c_1 = 1$ and $c_2 = 1$ and $c_3 \le 3$, and if $(c_1, c_2, c_3) = (1, 1, 2)$ then $c_4 \le 5$. Additionally, if $\gr{\CC} \le \xi$ then $c_n$ cannot be bounded below by any of the sequences in Tables 1 and 2 of \cite[Sec.\ 6]{PanVat}.

The proposition can now be proved by checking different cases based on how the sequence $\left(c_n\right)_{n\ge1}$ begins. By the discussion in the first paragraph of the proof, we can break and go to the next case any time the sequence goes below $(1, 1, 2, 3, 4^\infty)$. Along the way, \cite[Prop.\ 7.1]{PanVat} is used to eliminate the case where the sequence begins $(1, 1, 2, 3, 4^i, m)$ for $m \ge 6$, and \cite[Prop.\ 7.2]{PanVat} is used to show that the number of $4$s preceding the $5$ must be even or $1$.
\end{proof}

Proposition \ref{prop:bounded} shows that sum closed classes with growth rate $\le \xi$ satisfy condition (i) in the statement of Theorem \ref{main-sumclosed2}, so we obtain Theorem \ref{main-sumclosed} as an immediate corollary, which we restate here:
\begin{restate2}{main-sumclosed}
Let $\CC$ be a sum closed $rc$-invariant permutation class. If $\gr{\CC} \le \xi$, then $\grrc{\CC}$ exists and $\grrc{\CC} = \gr{\CC}$.
\end{restate2}

From inspecting Tables 3 and 4 in \cite[Sec.\ 8]{PanVat}, we immediately obtain:
\begin{prop} \label{prop:rational}
If $\CC$ is sum closed and $\gr{\CC} \le \xi$, then $\CC$ and $\ind{\CC}$ have rational generating functions.
\end{prop}

\begin{proof}
That $\ind{\CC}$ has a rational generating function when $\gr{\CC} < \xi$ is immediate from Tables 3 and 4 in \cite[Sec.\ 8]{PanVat}, because these tables list all possible sequences of numbers that $|\ind{\CC}_n|$ can give, and each sequence is eventually repeating (in fact it is shown in \cite[Sec.\ 7]{PanVat} that $|\ind{\CC}_3| \le 3$ implies that a sum closed class $\CC$ has a rational generating function). If $\gr{\CC} = \xi$, then the case checking in the proof of Proposition \ref{prop:bounded} also shows that $|\ind{\CC}_n|$ must be given by one of the following sequences:
\begin{itemize}
\item $(1, 1, 2, 4, 3, 3, 2, 1, 0^\infty)$;
\item $(1, 1, 2, 3, 4^\infty)$;
\item $(1, 1, 2, 3, 4^i, 5, 3, 3, 2, 1, 0^\infty)$;
\end{itemize}
and each of these has a rational generating function because it is eventually repeating.

That $\CC$ itself has a rational generating function is now immediate from Equation \eqref{eqn:ac}.
\end{proof}

%We end with a miscellaneous result on the proper growth rate of $\ind{\CC}$.
%
%\begin{prop} \label{zero-or-one}
%If $\upgr{\ind{\CC}}$ is $0$ or $1$, then the proper growth rate $\gr{\ind{\CC}}$ exists (and hence it equals $\upgr{\ind{\CC}}$).
%\end{prop}
%
%\begin{proof}
%If $|\ind{\CC}_n| \ge 1$ for all $n\ge1$, then $\logr{\ind{\CC}} \ge 1$, from which we obtain $\logr{\ind{\CC}} = \upgr{\ind{\CC}}$; hence $\gr{\ind{\CC}}$ exists and is $1$. Otherwise, there is $N\ge1$ such that $|\ind{\CC}_N| = 0$.
%
%Given $n \ge 2$, suppose that $|\ind{\CC}_n| \ge 1$; then there is $\pi \in \ind{\CC}_n$. It is known that $\pi$ has an entry whose deletion yields an indecomposable permutation, so from $\pi$ we can obtain a permutation in $\ind{\CC}_{n-1}$. Therefore, if $|\ind{\CC}_n| \ge 1$, then $|\ind{\CC}_k| \ge 1$ for all $k \le n$. So if there is $N\ge1$ such that $|\ind{\CC}_N| = 0$, then $|\ind{\CC}_n| = 0$ for all $n \ge N$. In this case, $\gr{\ind{\CC}}$ exists and is $0$.
%\end{proof}

\section{Open questions: New thresholds of growth rates} \label{sec:thresholds}

This section consists mostly of examples and conjectures, with the goal of beginning to answer two questions: What is the smallest possible growth rate of a sum closed class whose indecomposables are unbounded? And what is the smallest possible growth rate of a sum closed class whose indecomposables are exponential?

Let $R = \{1\} \cup \{ (12\ldots i) \ominus (12\ldots j) \colon i,j\ge1\}$, and let $\CC = \sumof R$. Every permutation in $R$ is indecomposable, and every indecomposable permutation contained in a permutation in $R$ is itself in $R$; consequently, $\ind{\CC} = R$. It turns out that $\CC = \Av{321, 3142, 2413}$. We have $|\ind{\CC}_n| = n-1$ for $n\ge2$, so
\[ \sum_{n\ge1} |\ind{\CC}_n|\,x^n = \frac{x-x^2+x^3}{(1-x)^2}, \]
and by Equation \eqref{eqn:ac} we obtain
\[ \sum_{n\ge1} |\CC_n|\,x^n = \frac{(1-x)^2}{1-3x+2x^2-x^3}. \]
From the denominator we see that $\gr{\CC}$ is the unique positive root of $x^3-3x^2+2x-1$, which is approximately $2.32472$. Call this number $\tau$.

\begin{conj} \label{conj:tau}
$\tau$ is the smallest possible growth rate of a sum closed class $\CC$ with the property that $|\ind{\CC}_n|$ is unbounded.
\end{conj}

Our example of $\CC = \Av{321, 3142, 2413}$ shows that $\tau$ is a possible growth rate, so the content of Conjecture \ref{conj:tau} is that no smaller growth rate is possible. In the other direction, we know from Proposition \ref{prop:bounded} that all such growth rates are greater than $\xi$. Thus, we know that the smallest threshold above which $|\ind{\CC}_n|$ can be unbounded is somewhere in the interval $[\xi, \tau]$, whose width is about $0.02$.

Conjecture \ref{conj:tau} would follow from this stronger conjecture:

\begin{conj} \label{conj:tau2}
If $\CC$ is sum closed and $|\ind{\CC}_n|$ is unbounded, then $|\ind{\CC}_n| \ge n-1$ for all $n$.
\end{conj}

Conjecture \ref{conj:tau2} would mean that, for each $n$, the smallest possible value of $|\ind{\CC}_n|$ is achieved by $\Avn{n}{321, 3142, 2413}$.

Now let $S = \{ \sigma \ominus 1 \colon \sigma \in \sumof \{1, 21\} \}$, and let $\CC = \sumof S$. Every permutation in $S$ is indecomposable, and every indecomposable permutation contained in a permutation in $S$ is itself in $S$; consequently, $\ind{\CC} = S$. It turns out that $\CC = \Av{312,4321,3421}$. Since $\{1, 21\}$ has generating function $x+x^2$, we find that $\sumof \{1,21\}$ has generating function $\frac{1}{1-x-x^2}$, and so $S = \ind{\CC}$ has generating function $\frac{x}{1-x-x^2}$, and by Equation \eqref{eqn:ac} we obtain
\[ \sum_{n\ge1} |\CC_n|\,x^n = \frac{1-x-x^2}{1-2x-x^2}. \]
From the denominator we see that $\gr{\CC}$ is the unique positive root of $x^2-2x-1$, which is $1+\sqrt{2} \approx 2.41421$.

In this example, we have $\gr{\ind{\CC}} = \phi = \smallfrac{1+\sqrt{5}}{2} \approx 1.61803$, which is greater than $1$. Is $1+\sqrt{2}$ the smallest growth rate of a class $\CC$ for which $\upgr{\ind{\CC}} > 1$? In the other direction, we know from Theorem \ref{main-sumclosed2} that $\gr{\CC} > \xi$ for all such classes. Thus, we know that the smallest threshold above which $\upgr{\ind{\CC}}$ can be greater than $1$  is somewhere in the interval $[\xi, 1+\sqrt{2}]$, whose width is about $0.11$.

There are no classes with upper or lower growth rate strictly between $1$ and $\phi$ \cite{KK}, but is there sum closed $\CC$ with $\upgr{\ind{\CC}}$ strictly between $1$ and $\phi$?

\subsection*{Acknowledgements}

I thank the two reviewers for helping to improve this paper, especially Reviewer 2 who provided background information on the motivation for our main question. I thank Jay Pantone for several helpful discussions about this paper, as well as for showing me the ``trick'' used to prove Theorem \ref{thm:monotone}. I also thank Michael Albert, David Bevan, and Robert Brignall (and a few others attending \textit{Permutation Patterns} 2018) for pointing out a very false ``theorem'' that was in an earlier draft of this paper. Finally, I thank Alexander Woo for presenting the open problem at \textit{Permutation Patterns} 2016 that led to this paper.


\begin{thebibliography}{10}

\bibitem{AABRV}
M.~H. Albert, M.~D. Atkinson, M.~Bouvel, N.~Ru\v{s}kuc, and V.~Vatter,
  Geometric grid classes of permutations, \emph{Trans.\ Amer.\ Math.\ Soc.}\
  \textbf{365} (2013), 5859--5881.

\bibitem{AAB}
M.~H. Albert, M.~D. Atkinson, and R.~Brignall, The enumeration of three
  pattern classes using monotone grid classes, \emph{Electron.\ J. Combin.}\
  \textbf{19(3)} (2012), \#P20.

\bibitem{AAV}
M.~H. Albert, M.~D. Atkinson, and V.~Vatter, Counting $1324,
  4231$-avoiding permutations, \emph{Electron.\ J. Combin.}\ \textbf{16(1)} (2009),
  \#R136.

\bibitem{Arratia}
R.~Arratia, On the {Stanley}-{Wilf} conjecture for the number of
  permutations avoiding a given pattern, \emph{Electron.\ J. Combin.}\ \textbf{6}
  (1999), \#N1.

\bibitem{Atkinson1998}
M.~D. Atkinson, Permutations which are the union of an increasing and a
  decreasing subsequence, \emph{Electron.\ J. Combin.}\ \textbf{5} (1998), \#R6.

\bibitem{Bevan}
D.~Bevan, Growth rates of geometric grid classes of permutations,
  \emph{Electron.\ J. Combin.}\ \textbf{21} (2014), \#P4.51.

\bibitem{BP}
S.~Billey and A.~Postnikov, Smoothness of {Schubert} varieties via
  patterns in root subsystems, \emph{Adv.\ in Appl.\ Math.}\ \textbf{34} (2005),
  447--466.

\bibitem{BJS}
S.~C. Billey, W.~Jockusch, and R.~P. Stanley, Some combinatorial
  properties of {Schubert} polynomials, \emph{J. Algebraic Combin.}\ \textbf{2}
  (1993), 345--374.

\bibitem{BHPV}
M.~B\'ona, C.~Homberger, J.~Pantone, and V.~Vatter, Pattern-avoiding
  involutions: exact and asymptotic enumeration, \emph{Australas.\ J. Combin.}\
  \textbf{64} (2016), 88--119.

\bibitem{CJS}
A.~Claesson, V.~Jel\'inek, and E.~Steingr\'imsson, Upper bounds for the
  {Stanley}--{Wilf} limit of $1324$ and other layered patterns, \emph{J. Combin.\
  Theory Ser.\ A} \textbf{119} (2012), 1680--1691.

\bibitem{Egge}
E.~S. Egge, Restricted symmetric permutations, \emph{Ann.\ Comb.}\ \textbf{11}
  (2007), 405--434.

\bibitem{Egge2}
E.~S. Egge, Enumerating $rc$-invariant permutations with no long decreasing
  subsequences, \emph{Ann.\ Comb.}\ \textbf{14} (2010), 85--101.

\bibitem{ElizaldeX}
S.~Elizalde, The {$X$}-class and almost-increasing permutations, \emph{Ann.\
  Comb.}\ \textbf{15} (2011), 51--68.

\bibitem{KK}
T.~Kaiser and M.~Klazar, On growth gates of closed permutation classes,
  \emph{Electron.\ J. Combin.}\ \textbf{9(2)} (2003), \#R10.

\bibitem{Knuth}
D.~E. Knuth, \emph{The art of computer programming vol.\ 1}, Addison--Wesley,
  Reading, MA, 1968.

\bibitem{KS}
D.~Kremer and W.~C. Shiu, Finite transition matrices for permutations
  avoiding pairs of length four patterns, \emph{Discrete Math.}\ \textbf{268} (2003),
  171--183.

\bibitem{LO}
D.~Lonoff and J.~Ostroff, Symmetric permutations avoiding two patterns,
  \emph{Ann.\ Comb.}\ \textbf{14} (2010), 143--158.

\bibitem{MacMahon}
P.~A. MacMahon, \emph{Combinatory analysis}, vol.~1, Cambridge Univ.\ Press,
  London, 1915.

\bibitem{MarcusTardos}
A.~Marcus and G.~Tardos, Excluded permutation matrices and the
  {Stanley}--{Wilf} conjecture, \emph{J. Combin.\ Theory Ser.\ A} \textbf{107} (2004),
  153--160.

\bibitem{PanVat}
J.~Pantone and V.~Vatter, Growth rates of permutation classes:
  categorization up to the uncountability threshold, \emph{Israel J. Math.}\ (to
  appear), \href{https://arxiv.org/abs/1605.04289}{\texttt{arXiv:1605.04289}}.

\bibitem{Regev}
A.~Regev, Asymptotic values for degrees associated with strips of {Young}
  diagrams, \emph{Adv.\ Math.}\ \textbf{41} (1981), 115--136.

\bibitem{SS}
R.~Simion and F.~R. Schmidt, Restricted permutations, \emph{European J.
  Combin.}\ \textbf{6} (1985), 383--406.

\bibitem{Stankova}
Z.~Stankova, Classification of forbidden subsequences of length $4$,
  \emph{European J. Combin.}\ \textbf{17} (1996), 501--517.

\bibitem{Stembridge}
J.~R. Stembridge, Some combinatorial aspects of reduced words in finite
  {Coxeter} groups, \emph{Trans.\ Amer.\ Math.\ Soc.}\ \textbf{349} (1997), 1285--1332.

\bibitem{Vatter2012}
V.~Vatter, Finding regular insertion encodings for permutation classes,
  \emph{J. Symbolic Comput.}\ \textbf{47} (2012), 259--265.

\bibitem{survey}
V.~Vatter, Permutation classes, \emph{Handbook of Enumerative Combinatorics},
  \emph{Discrete Math.\ Appl.\ (Boca Raton)}, CRC Press, Boca Raton, FL, 2015.

\bibitem{West}
J.~West, Permutations with forbidden subsequences; and, stack-sortable
  permutations, Ph.D. thesis, Massachusetts Institute of Technology, 1990.

\bibitem{West1995}
J.~West, Generating trees and the {Catalan} and {Schr\"oder} numbers,
  \emph{Discrete Math.}\ \textbf{146} (1995), 247--262.

\bibitem{West1996}
J.~West, Generating trees and forbidden subsequences, \emph{Discrete Math.}\
  \textbf{157} (1996), 363--374.

\bibitem{Woo}
A.~Woo, Hultman elements for the hyperoctahedral groups, \emph{Electron.\ J.
  Combin.}\ \textbf{25} (2018), \#P2.41.

\end{thebibliography}
\end{document}